\documentclass[a4paper,12pt]{amsart}

\usepackage{float}
\usepackage{euscript,eufrak,verbatim}
\usepackage{graphicx}
\usepackage[usenames]{color}
\usepackage[colorlinks,linkcolor=red,anchorcolor=blue,citecolor=blue]{hyperref}
\usepackage{amsmath}
\usepackage{amsthm}
\usepackage[all]{xy}
\usepackage{graphicx} 
\usepackage{amssymb}

\newtheorem{thm}{Theorem}[section]
\newtheorem{prop}[thm]{Proposition}

\newtheorem{lem}[thm]{Lemma}
\newtheorem{cor}[thm]{Corollary}

\newtheorem{problem}[thm]{Problem}

\newtheorem{conj}[thm]{Conjecture}

\def\N{\mathbb{N}}
\def\Z{\mathbb{Z}}
\def\N{\mathbb{N}}
\def\R{\mathbb{R}}

\def\XX{\mathcal{X}}
\def\YY{\mathcal{Y}}
\def\ZZ{\mathcal{Z}}

\def\SS{\mathcal{S}}

\def\SS{\mathcal{S}}

\def\mdim{\text{\rm mdim}}
\def\AA{\mathcal{A}}
\def\BB{\mathcal{B}}
\def\coindex{\text{\rm coind}}
\def\coindexPER{\text{\rm coind}^{\rm Per}}

\makeatletter 
\@addtoreset{equation}{section}
\numberwithin{equation}{section}

\title{Finite mean dimesnion and marker property}
\author{Ruxi Shi
}
\address
{Institute of Mathematics, Polish Academy of Sciences, ul. \'Sniadeckich 8, 00-656 Warszawa, Poland}
\email{rshi@impan.pl}
\begin{document}
	\maketitle
\begin{abstract}
    In this paper, we develop the theory of $\mathbb{Z}_p$-index which has been introduced by Tsukamoto, Tsutaya and Yoshinaga.
    As an application, we show that given any positive number, there exists a dynamical system with mean dimension equal to such number such that it does not have the marker property.
\end{abstract}

\section{Introduction}

A \textit{(topological) dynamical system} is a pair $(X,T)$ where $X$ is compact metrizable space and $T: X\to X$ is a homeomorphism. The most basic invariant of dynamical systems is topological
entropy which has been studied for a long time. Gromov \cite{G} introduced a new topological invariant of dynamical systems called {\it mean topological dimension} (see the precise definition in Section \ref{sec:Preliminaries}). It was further developed by Lindenstrauss and Weiss \cite{LinWeiss2000MeanTopologicalDimension}. Roughly speaking, mean topological dimension captures the complexity of dynamical systems of infinite entropy.

For a dynamical system $(X, T)$, let $P_n(X,T)$ be the set of $n$-periodic points, i.e.
$$
P_n(X,T):=\{x\in X: T^nx=x \}.
$$
The dynamical system $(X,T)$ is said to be \textit{aperiodic} if $P_n(X,T)=\emptyset$ for all $n\ge 1$. For a dynamical system without fixed points (i.e. $P_1=\emptyset$), Tsukamoto, Tsutaya and Yoshinaga \cite{tsukamoto2020markerproperty} introduced $\Z_p$-index to
study  the set of $p$-periodic points for prime number $p$. They showed its growth is at most linear in $p$.
Moreover, they applied the theory of $\Z_p$-index to study the marker property.
For a dynamical system $(X,T)$, it is said to satisfy the {\it marker property} if for each positive integer $N$ there exists an open set $U\subset X$ satisfying that
$$
U\cap T^{-n}U=\emptyset~\text{for}~0<n<N~\text{and}~X=\cup_{n\in \Z} T^{n}U.
$$
For example, an extension of an aperiodic minimal system has
the marker property \cite[Lemma 3.3]{L99}. Obviously, the marker property implies the aperiodicity. 

The marker property has been extensively used in the theory of mean dimension.  Lindenstrauss \cite{L99} proved that for a dynamical system having the marker property, it has zero mean dimension if and only if it is isomorphic to an inverse limit of finite entropy dynamical systems.
Gutman, Qiao and Tsukamoto \cite{gutman2019application} showed that a dynamical system having the marker property with mean dimension smaller than $N/2$ can be embedded in the shift of Hilbert cube $([0,1]^N)^\Z$. Lindenstrauss and Tsukamoto \cite{lindenstrauss2019double} proved that if a dynamical system has the marker property, then one can find a metric such that the upper metric mean  dimension is equal to mean dimension of this dynamical system. The {\it metric mean  dimension} was introduced by Lindenstrauss and Weiss \cite{LinWeiss2000MeanTopologicalDimension} which  majors the value of mean topological dimension. See also \cite{Gut15Jaworski,GutLinTsu15, gutman2017embedding,gutman2020embedding,tsukamoto2020potential} and references therein where the marker property was investigated. 

As we mentioned before,  the marker property implies the aperiodicity. For several years, it has been specially curious whether the converse is true, i.e. whether the aperiodicity implies the marker property. This open problem was stated by Gutman (\cite[Problem 5.4]{Gut15Jaworski}, \cite[Problem 3.4]{gutman2017embedding}) and its origin goes back to the paper of Lindenstrauss \cite{L99}. Very recently, Tsukamoto, Tsutaya and Yoshinaga \cite{tsukamoto2020markerproperty} solved this problem by constructing a counter-example of an aperiodic dynamical system with infinite mean dimension which does not have the marker property. It is very surprising that such a counter-example comes from an application of $\Z_p$-index theory which they studied in the same paper. Since the example that they constructed is infinite mean dimensional, they proposed the following problem  \cite[Problem 7.6]{tsukamoto2020markerproperty}.

\begin{problem}\label{prob}
Construct an aperiodic dynamical system with finite mean dimension which does
not have the marker property.
\end{problem}
As Tsukamoto, Tsutaya and Yoshinaga commented in their paper, they believed that such an example exists but it seems difficult to construct such example by their current method.

In this paper, we continue to investigate $\Z_p$-index theory and its connection to dynamical systems. As an application, we give an affirmative answer to Problem \ref{prob}.  Furthermore, we show that the desired dynamical system that we constructed could have mean dimension arbitrarily close to zero. Our main result is stated as follows.
\begin{thm}\label{main thm}
For any $\eta>0$, there exists an aperiodic dynamical system with mean dimension small than $\eta$ which does not have the marker property. 
\end{thm}

Now we present the main idea and explain what are the main difficulties in the proof Theorem \ref{main thm}. There are two principal steps in the proof. 

Step 1. We first find a family of dynamical systems with same periodic coindex (see the precise definition in Section \ref{sec:Coindex}) and uniformly bounded mean dimension from above which has the universal property (see the discussion in Section \ref{sec:Universal dynamical system}). In fact, these dynamical systems turn out to be subshifts of $((\R/2\Z)^N)^\Z$ for some positive number $N$. We choose a sequence of dynamical systems $(X_n, T_n)$ in such family satisfying that the smallest period is getting larger and larger as $n$ tends to infinity.  
Then we show that the inverse limit of $\{(X_n, T_n)\}_{n\in \mathbb{N}}$ is an aperiodic dynamical system with finite mean dimension which does not have the marker property (Theorem \ref{thm:finite mean dimension}). This gives a positive answer to Problem \ref{prob}. The main difficulties that we face in the above construction is the following 
two issues. The first one is that we need certain group structure in the base space of subshifts. This is the reason why we can not apply directly the universal dynamical systems studied in \cite{tsukamoto2020markerproperty} to our proof. The second one is that we need a continuous left inverse map of the nature projection of the inverse limit (see the defintion of $\gamma_m$ in Section \ref{sex:Finite mean dimension}). The disadvantage of using the inverse limit is that we can not expect the result (Lemma \ref{lem:to Z}) similar to \cite[Lemma 6.1]{tsukamoto2020markerproperty} holding for all inverse limits (unlike \cite[Lemma 6.1]{tsukamoto2020markerproperty} holds for all infinite product). However, the advantage is that the mean dimension of the inverse limit is controlled by the superior of the mean dimensions of $(X_n, T_n)$ (See Proposition \ref{prop:mean dimension of inverse limit}). 

Step 2. Based on the dynamical system $(X,T)$ (i.e. the inverse limit of $\{(X_n, T_n)\}_{n\in \mathbb{N}}$ constructed in Step 1), we build a $\frac{1}{n}$-time system $(X\times \Z_n, T_n)$ (see the precise definition in Section \ref{sec:Marker property and suspension flow}) such that it has the marker property for every $n\in \N$. Since the mean dimension of $\frac{1}{n}$-time system $(X\times \Z_n, T_n)$ is the mean dimension of $(X, T)$ divided by $n$, we get the desired result.


By Theorem \ref{main thm} (and as well as Corollary \ref{cor:any mean dimension example}), one can not expect that the marker property implies aperiodicity for dynamical systems with positive mean dimension. The zero mean dimensional ones is the only case left. It seems to be believable that the marker property and zero mean dimension together imply aperiodicity. In fact, this is equivalent to a conjecture of Lindenstrauss \cite{L99}.
See the discussion in Section \ref{sec:open problem}.

The paper is organized as follows. In Section \ref{sec:Preliminaries}, we recall the definitions of inverse limit and mean dimension. In Section \ref{sec:Coindex}, we recall the $\Z_p$-coindex of dynamical system which was introduced in \cite{tsukamoto2020markerproperty}. In Section \ref{sec:Universal dynamical system}, we construct a universal dynamical system without fixed points. In Section \ref{sex:Finite mean dimension}, we construct an aperiodic dynamical system with finite mean dimension which does not have marker property. In Section \ref{sec:Marker property and suspension flow}, we construct an aperiodic dynamical system without marker property whose mean dimension is arbitrarily small. In Section \ref{sec:open problem}, we discuss open problems and conjectures.

\section{Preliminaries}\label{sec:Preliminaries}

For dynamical systems $(X,T)$ and $(Y,S)$, a map $f:X\to Y$ is said to be \textit{equivariant} if $f\circ T=S\circ f$. 
A dynamical system $(X, T)$ is called a {\it topological factor} of a system $(Y,S)$ if there exists a continuous and equivariant surjection $\pi: X\to Y$. The map $\pi$ is called a {\it factor map}.

\subsection{Inverse limit of dynamical systems}

Let $\{(X_n, T_n) \}_{n\in \N}$ be a sequence of dynamical systems. Suppose for every pair $m>n$ there exists a factor map $\sigma_{m,n}: X_m\to X_n$ such that for any triple $m>n>l$ the diagram
\begin{equation*}
\xymatrix{
	& (X_m, T_m) \ar[d]^{\sigma_{m,n}}   \ar[ld]_{\sigma_{m,l}}    \\
	(X_l, T_l)  &  \ar[l]^{\sigma_{n,l}} (X_n, T_n) 
}
\end{equation*}
commutes. Set
$$
X=\{(x_n)_{n\in \N}\in \prod_{n\in \N} X_n: \sigma_{m,n}(x_m)=x_n, \forall~m>n \}.
$$
Clearly, the space $X$ is closed in $\prod_{n\in \N} X_n$ and thus compact.
Let $\pi_n:X\to X_n$ be the natural projection map for each $n\in \N$. Define the continuous equivariant map $T:X\to X$ by
$$
T: (x_n)_{n\in \N} \to (T_nx_n)_{n\in \N}.
$$ 
It satisfies that $\sigma_{m,n}\circ \pi_m=\pi_n$ for all $m>n$.
We call the dynamical system $(X,T)$ the {\it inverse limit} of the family $\{(X_n, T_n) \}_{n\in \N}$ via $\sigma=(\sigma_{m,n})_{m,n\in \N, m>n}$ and write 
$$
(X,T)=\varprojlim (X_n,T_n).
$$

\subsection{Mean dimension}
Let $X$ be a compact space. For two finite open covers $\mathcal{A}$ and $\mathcal{B}$ of $X$, we say that the cover $\mathcal{B}$ is \textit{finer} than the cover $\mathcal{A}$, and write $\mathcal{B}\succ \mathcal{A}$, if for every element of $\mathcal{B}$, one can find an element of $\mathcal{A}$ which contains it.  For a finite open cover $\mathcal{A}$  of $X$, we define the quantities 
$$
\text{\rm ord}(\mathcal{A}):=\sup_{x\in X} \sum_{A\in \mathcal{A}} 1_A(x)-1,
$$
and
$$
D(\mathcal{A}):=\min_{\mathcal{B}\succ \mathcal{A}} \text{ord}(\mathcal{B}).
$$
Clearly, if $\mathcal{B}\succ \mathcal{A}$ then $D(\mathcal{B})\ge D(\mathcal{A})$. The \textit{(topological) dimension} of $X$ is defined by
$$
\text{dim}(X):=\sup_{\mathcal{A}} D(\mathcal{A}),
$$
where $\mathcal{A}$ runs over all finite open covers of $X$. For finite open covers $\mathcal{A}$ and $\BB$, we set $\AA \vee \BB:=\{U\cap V: U\in \AA, V\in \BB \}$. It is easy to check that $D(\AA\vee \BB)\le D(\AA)+D(\BB)$.
 
Let $(X, T)$ be a dynamical system. The {\it mean dimension} of $(X, T)$ is defined by
$$
\mdim(X, T)=\sup_{\alpha} \lim_{n\to \infty} \frac{D(\vee_{i=0}^{n-1} T^{-i}\alpha )}{n},
$$
where $\mathcal{A}$ runs over all finite open covers of $X$. The limit above exists is due to the sub-additivity of $D$.

We mention some basic properties of $D$ and mean dimension. Refer to the book \cite{Coo05} for the proofs and further properties. \begin{itemize}
    \item Let $X$ and $Y$ be topological spaces and let $f: X\to Y$ be a continuous map. Let $\AA$ be a finite open cover of $Y$. Then $D(f^{-1}(\AA))\le D(\AA)$.
    \item If $(Y,T)$ is a subsystem of $(X,T)$, i.e. $Y$ is a closed $T$-invariant subspace of $X$, then $\mdim(Y,T)\le \mdim(X,T)$.
    \item $\mdim(X,T^n)=n\cdot \mdim(X,T)$.
\end{itemize}

\section{Coindex of free $\Z_p$ system}\label{sec:Coindex}
We denote by $\Z_p:=\Z/p\Z$ for prime number $p$ through this paper. A pair $(X, T)$ is called a {\it $\Z_p$-system} if $(X,T)$ is a dynamical system and $T$ induces a $\Z_p$-action on $X$. Moreover, the $\Z_p$-system is said to be {\it free} if it has no fixed points, i.e. $Tx\not=x$ for all $x\in X$. Since $p$ is prime, a $\Z_p$-system is free if and only if it is aperiodic, i.e.  $T^nx\not=x$ for all $1\le n\le p-1$ and $x\in X$.

A free $\Z_p$-system $(X, T)$ is called an {\it $E_n\Z_p$-system} if it satisfies the following conditions:
\begin{itemize}
    \item $X$ is an $n$-dimensional finite simplicial complex.
    \item $T:X\to X$ is a simplical map, i.e. it maps every simplex to a simplex affinely.
    \item $X$ is $(n-1)$-connected, i.e. the $k$-th homotopy group $\pi_k(X)=0$ for all $k\le n-1$. 
\end{itemize}
For example, if $\Z_p$ acts freely by $T$ on a finite set $Y$, then $(Y, T)$ is an $E_0\Z_p$-system. Moreover, if $Y_i$ are finite sets and $(Y_i, T_i)$ are free $\Z_p$-systems for $0\le i\le n$, then $(Y_0*Y_1*\cdots *Y_n, T_0*T_1*\cdots *T_n)$ is an $E_n\Z_p$-system. Recall that for  dynamical systems $(X, T)$ and $(Y, S)$, the dynamical systems $(X*Y, T* S)$ is defined by 
$$X*Y:=[0,1]\times X\times Y/\sim ~\text{and}~ T*S(t,x,y)=(t,Tx,Sy)$$ where the equivalence relation $\sim$ is given by
$$
(0, x, y)\sim (0,x, y')~\text{and}~(1, x, y)\sim (1,x', y),
$$
for any $x,x'\in X$ and any $y,y'\in Y$. It is easy to see that 
$$
P_n(X*Y, T* S)=P_n(X, T)*P_n(Y, S), \forall n\in \N.
$$
In particular, if $(X, T)$ and $(Y, S)$ are both aperiodic, then so is $(X*Y, T* S)$.

Due to \cite[Lemma 6.2.2]{matouvsek2003using}, there is an equivariant and continuous map from one $E_n\Z_p$-system to another. Thus, the $E_n\Z_p$-system is ``unique" in the sense that we regard two systems as the same whenever there are equivariant continuous maps between each other.

Following \cite{tsukamoto2020markerproperty}, we define the {\it coindex} of a free $\Z_p$-system $(X,T)$ by 
$$
\coindex_p (X,T):=\max\{n\ge 0: \exists~\text{an equivariant continuous}~ f: E_n\Z_p \to X  \}
$$
We  use the convention that $\coindex_p(X,T)=-1$ for $X=\emptyset$.
Notice that if there exists an equivariant continuous map from $E_m\Z_p$ to $E_n\Z_p$ then $m\le n$ (see \cite[Theorem 6.2.5]{matouvsek2003using}). Thus 
$$
\coindex_p (E_n\Z_p, T)=n.
$$
Moreover, if $(X,T)$ is an $E_n\Z_p$-system, then $\coindex_p(X,T)=n$.

\begin{lem}\label{lem:coprime power}
Let $(X, T)$ be a free $\Z_p$-system. For any integer $n$ coprime with $p$, we have
$$
\coindex_p(X, T^n)=\coindex_p(X,T).
$$
\end{lem}
\begin{proof}
It follows directly from the fact that for any integer $n$ coprime with $p$ if $(Y,S)$ is an $E_n\Z_p$-system then $(Y,S^n)$ is not only an $E_n\Z_p$-system but also topologically conjugate to $(Y,S)$.
\end{proof}


The following basic properties of coindex were proved by Tsukamoto, Tsutaya and Yoshinaga. 
\begin{prop}[\cite{tsukamoto2020markerproperty}, Proposition 3.1]\label{prop:basic property of coindex}
Let $(X, T)$ and $(Y, S)$ be free $\Z_p$-systems. Then the following properties hold.
\begin{itemize}
    \item [(1)] If there is an equivariant continuous map $f: X\to Y$ then $\coindex_p(X,T)\le \coindex_p(Y, S)$.
    \item [(2)] The system $(X*Y, T* S)$ is a free $\Z_p$-systems and $\coindex_p(X*Y, T* S)\ge \coindex_p(X,T)+\coindex_p(Y, S)+1$.
\end{itemize}
\end{prop}

Let $(X,T)$ be a  dynamical system $(X,T)$ without fixed points (i.e. $Tx\not=x$ for all $x\in X$). Recall that $P_n(X,T)$ is the set of $n$-periodic points.
Observe that $(P_p(X, T), T)$ is a free $\Z_p$-system for prime number $p$. We define the \textit{periodic coindex} of $(X,T)$ as 
$$
\coindex_p^{\rm Per}(X,T)=\coindex_p((P_p(X,T)), T),
$$
for prime numbers $p$.
\begin{cor}\label{cor:basic property}
Let $(X, T)$ and $(Y, S)$ be dynamical systems without fixed points. Let $p$ be a prime number. Then the following properties hold.
\begin{itemize}
    \item [(1)] If there is an equivariant continuous map $f: X\to Y$ then $\coindexPER_p(X,T)\le \coindexPER_p(Y, S)$.
    \item [(2)] The system $(X*Y, T* S)$ has no fixed points and $\coindexPER_p(X*Y, T* S)\ge \coindexPER_p(X,T)+\coindexPER_p(Y, S)+1$.
\end{itemize}
\begin{proof}
(1) Notice that $f$ restricted to the set of $p$-periodic points induces an equivariant continuous map from $(P_p(X,T), T)$ to $(P_p(Y,S), S)$. Then we apply Proposition \ref{prop:basic property of coindex} (1).

(2) Observe the fact that $P_p(X*Y, T* S)=P_p(X, T)*P_p(Y, S)$. Then we apply Proposition \ref{prop:basic property of coindex} (2).
\end{proof}
\end{cor}

It follows from Corollary \ref{cor:basic property} (1) that the sequence $(\coindexPER_p(\cdot))_p$ is a topological invariant (i.e. it is invariant under topological conjugacy) in the catalogue of dynamical systems without fixed points. Moreover, if we use the convention that  $\coindexPER_p(\cdot)=-2$ for dynamical systems with fixed point then $(\coindexPER_p(\cdot))_p$ becomes a topological invariant of dynamical systems.

\section{Universal dynamical system without fixed points}\label{sec:Universal dynamical system}
Let $\SS=\R/2\Z$. In what follows, we also regard $\SS$ as the interval $[-1,1]$ such that the endpoints $-1$ and $1$ are identified with each other. Let $\rho$ be the metric on $\SS$ defined by
$$
\rho(x, y)=\min_{n\in \Z} |x-y-2n|.
$$
It is easy to see that the diameter of $\SS$ under the metric $\rho$ is $1$ and $\rho$ is {\it homogeneous}, i.e. $\rho(x+z, y+z)=\rho(x, y)$. 

Let $N$ be a positive integer. Let $\rho_N$ be the metric on $\SS^N$ defined by
$$
\rho_N(x, y)=\max_{1\le i\le N} \rho(x^i, y^i),
$$
where $x=(x^i)_{1\le i\le N}$, $y=(y^i)_{1\le i\le N}\in \SS^N$. Then the diameter of $\SS^N$ under the metric $\rho_N$ is $1$. Notice that $\SS^N$ is an abelian group and $\rho_N(x+z, y+z)=\rho_N(x, y)$ for any $x,y,z\in \SS^N$ (because $\rho(x^i+z^i, y^i+z^i)=\rho(x^i, y^i)$ as we have mentioned before). The product space $(\SS^N)^\Z$ is compact  and metrizable under the product topology and the metric $d(x, y)=\sum_{n\in \Z} \frac{1}{2^n} \rho_N(x_n, y_n)$ where $x=(x_n)_{n\in \N}$ and $y=(y_n)_{n\in \N}$.
Let $\sigma: (\SS^N)^\Z \to (\SS^N)^\Z$ be the shift map, i.e. $\sigma: (x_n)_{n\in \Z} \to (x_{n+1})_{n\in \Z}$. In this paper, we always use $\sigma$ to denote the shift map on $(\SS^N)^\Z$ and as well as its subshift.


Tsukamoto, Tsutaya and Yoshinaga \cite{tsukamoto2020markerproperty} studied the universal property of certain subsystem in Hilbert cube $([0,1]^N)^\Z$.  Notice that $[0,1]$ can be embedded into $\SS$ by $x\mapsto \frac{1}{2}x$; conversely, $\SS$ can be embedded into $[0,1]^2=\{x+yi: x,y\in [0,1]\}$ by $x\mapsto \frac{1}{2}e^{\pi i x}+e^{\frac{1}{2}+\frac{1}{2}i}$.  
It is reasonable to expect that some similar universal property is valid for certain subshifts of $(\SS^N)^\Z$. We investigate such universal property in this section. 

A continuous map $f$ from a metric space $(X,d)$ to a space $Y$ is called an {\it $\epsilon$-embedding} if $f(x)=f(y)$ implies $d(x, y)<\epsilon$.
\begin{lem}\label{lem:epsilon embedding}
Let $\epsilon>0$. Let $(X, d)$ be a compact metric space. There exists an integer $N>0$ and an $\epsilon$-embedding from $X$ to $\SS^N$.
\end{lem}
\begin{proof}
Without loss of generality, we assume diam$(X)\le 1/4$. Pick an open cover $\{B(x_i, \frac{\epsilon}{2})\}_{1\le i\le N}$ of $X$ where $B(x, \frac{\epsilon}{2})$ is the open ball centered at $x$ of radius $\frac{\epsilon}{2}$. Define a continuous map $f: X \to \SS^N$ by
$$
f(x)=(d(x, x_1), d(x, x_2), \dots, d(x, x_N)).
$$
For $x,y\in X$ with $f(x)=f(y)$, supposing $x\in B(x_i, \frac{\epsilon}{2})$, we obtain that $d(y, x_i)=d(x,x_i)<\epsilon/2$ and thus $d(x,y)<\epsilon$. This implies that $f$ is $\epsilon$-embedding. 
\end{proof}

For any positive integers $m, N$ and any number $\delta>0$, we define a subsystem $(\XX(N, m, \delta), \sigma)$ of $((\SS^N)^\Z, \sigma)$ by
$$
\XX(N, m, \delta):=\{(x_n)_{n\in \Z}\in (\SS^N)^\Z: \rho_N(x_n, x_{n+m})\ge \delta, ~\forall n\in \Z \}.
$$
It is clear that $(\XX(N, m, \delta), \sigma)$ has no fixed points. The universal property of $\XX(N, 1, \delta)$ is studied in the following lemma.

\begin{lem}\label{lem:in N,1,delta}
Let $(X,T)$ be a dynamical system without fixed points. Then there exists an positive integer $N$, an $\delta>0$ and an equivariant continuous map from $X$ to $\XX(N, 1, \delta)$.
\end{lem}
\begin{proof}
Let $d$ be a metric on $X$ which is compatible with its topology. Since $(X, T)$ has no fixed point, we have $\inf_{x\in X}d(x, Tx)>0$. Pick $0<\epsilon<\inf_{x\in X}d(x, Tx)$. By Lemma \ref{lem:epsilon embedding}, there is an integer $N$ and an $\epsilon$-embedding $f$ from $X$ to $\SS^N$. Define an equivariant and continuous map $g: X\to (\SS^N)^\Z$ by
$$
x\mapsto (f(T^nx))_{n\in \Z}.
$$
Since $X$ is compact and $f$ is a $\epsilon$-embedding, there is an $\delta>0$ such that $|f(x)-f(y)|<\delta$ implies $d(x,y)<\epsilon$. Then we get that $$|f(x)-f(Tx)|\ge \delta$$ for all $x\in X$. It follows that the image of $X$ under $g$ is contained in $\XX(N, 1, \delta)$. This completes the proof.
\end{proof}

Even though we can not replace $\XX(N, 1, \delta)$ in Lemma \ref{lem:in N,1,delta} by $\XX(N, m, \delta)$ for $m\ge 2$, the following lemma tells us that $\XX(N, m, \delta)$ is the same as $\XX(N, 1, \delta)$ in terms of periodic coindex for large prime numbers. 

\begin{lem}\label{lem:m are the same coindex}
Let $0<\delta<1$.
 For any positive integer $m$ and any prime number $p>m$, we have
$$
\coindexPER_p (\XX(N,m, \delta), \sigma)= \coindexPER_p (\XX(N,1, \delta), \sigma)\ge 0.
$$
\end{lem}
\begin{proof}
Fix a positive integer $m$ and a prime number $p>m$. 
Notice that the set of $p$-periodic points
$$
P_p(\XX(N, m, \delta))=\{(x_i)_{i\in \Z_p}\in (\SS^N)^{\Z_p}: \rho_N(x_n, x_{n+m})\ge \delta, ~\forall n\in \Z_p \},
$$
which is not empty.
Define $f_j: (\SS^N)^{\Z_p} \to (\SS^N)^{\Z_p}$ by $(x_i)_{i\in \Z_p} \mapsto (x_{ij})_{i\in \Z_p}.$
Let $1\le k\le p-1$ be the inverse of $m$ in $\Z_p$, i.e. $km\equiv 1 \mod p$. Then the following diagram commutes:
\begin{equation*}
\xymatrix{
	& P_p(\XX(N, m, \delta), \sigma) \ar[r]^{\sigma^m} \ar@<-.5ex>[d]^{f_m}  & P_p(\XX(N, m, \delta), \sigma)   \ar@<-.5ex>[d]_{f_k}     \\
	& P_p(\XX(N, 1, \delta), \sigma) \ar[r]^{\sigma} \ar[u]^{f_k}  & P_p(\XX(N, 1, \delta), \sigma)  \ar[u]_{f_m}
}
\end{equation*}
It means that the system $(P_p(\XX(N, m, \delta), \sigma), \sigma^m) $ is topologically conjugate to $(P_p(\XX(N, 1, \delta), \sigma), \sigma) $. Then by Lemma \ref{lem:coprime power}, we obtain that
\begin{equation*}
    \begin{split}
        \coindexPER_p(\XX(N, 1, \delta), \sigma)
        &=\coindex_p(P_p(\XX(N, m, \delta), \sigma), \sigma^m)\\
        &=\coindex_p(P_p(\XX(N, m, \delta), \sigma), \sigma)\\
        &=\coindexPER_p(\XX(N, m, \delta), \sigma).
    \end{split}
\end{equation*}
This completes the proof.
\end{proof}

\begin{lem}\label{lem:dynamical system with vashing coindex}
There exists a dynamical system $(Y, S)$ satisfying that $P_1(Y,S)=\emptyset$ and $0<\sharp ( P_n(Y, S) )<+\infty$ for $n\ge 2$. Moreover, we have $\coindexPER_p(Y,S)=0$ for all prime number $p$.
\end{lem}
\begin{proof}
Let $Y:=\{ (x_n)_{n\in \N}\in \{0,1\}^\Z: x_nx_{n+1}x_{n+2}\not= 000 ~\text{or}~111\}$. Let $S: Y\to Y$ be the shift. Then it is easy to check that $(Y,S)$ has no fixed points and $0<\sharp P_n(Y, S)<+\infty$ for $n\ge 2$.
Since $P_n(Y, S)$ is non-empty and finite for $n\ge 2$, we have $\coindexPER_p(Y,S)=0$ for all prime number $p$.
\end{proof}
See another example of above lemma in \cite[Lemma 4.1]{tsukamoto2020markerproperty}
\begin{prop}\label{prop:universal coindex}
Let $(X,T)$ be a dynamical system without fixed points. Then there exists a positive integer $N$ and a positive number $\delta$ such that 
$$
\coindexPER_p(\XX(N, 1, \delta), \sigma)\ge \coindexPER_p(X,T)+1,
$$
for all prime numbers $p$.
\end{prop}
\begin{proof}
By Lemma \ref{lem:dynamical system with vashing coindex}, we pick a dynamical system $(Y, S)$ without fixed points satisfying that $\coindexPER_p(Y,S)=0$ for all prime number $p$. Then by Corollary \ref{cor:basic property} (2), we have
\begin{equation}\label{eq:1}
\begin{split}
    \coindexPER_p(X*Y, T*S)
    &\ge \coindexPER_p(X, T)+\coindexPER_p(Y, S)+1\\
    &=\coindexPER_p(X, T)+1,
\end{split}
\end{equation}
for all prime number $p$.
On the other hand, by Lemma \ref{lem:in N,1,delta}, there exists a dynamical system $(\XX(N, 1, \delta), \sigma)$ and an equivariant continuous map $$f: (X*Y, T*S) \to (\XX(N, 1, \delta), \sigma).$$
By Corollary \ref{cor:basic property} (1), we have
\begin{equation}\label{eq:2}
    \coindexPER_p(X*Y, T*S)\le \coindexPER_p(\XX(N, 1, \delta), \sigma),
\end{equation}
for all prime number $p$.
Combing \eqref{eq:2} with \eqref{eq:1}, we conclude that 
$$
\coindexPER_p(\XX(N, 1, \delta), \sigma)\ge \coindexPER_p(X,T)+1,
$$
for all prime number $p$.
\end{proof}

\section{Finite mean dimension}\label{sex:Finite mean dimension}

In this section, we show the existence of aperiodic dynamical systems with finite mean dimension which do not have the marker property. We state our result as follows.
\begin{thm}\label{thm:finite mean dimension}
There exists an aperiodic dynamical system with finite mean dimension which does not have the marker property. 
\end{thm}

Recall that the diameter of $\SS$ is $1$. Define a subsystem of $(\SS^\Z, \sigma)$ by 
$$
\ZZ:=\{(x_n)_{n\in \Z} \in \SS^\Z : \forall n\in \Z, ~\text{either}~\rho(x_{n-1}, x_{n})\ge \frac{1}{2} ~\text{or}~ \rho(x_{n}, x_{n+1})\ge \frac{1}{2} \}.
$$
Obviously,  the dynamical system $(\ZZ, \sigma)$ has no fixed points and $P_p(\ZZ, \sigma)\not=\emptyset$ for all prime numbers $p$. By Proposition \ref{prop:universal coindex}, there exists a positive integer $N$, an $\delta>0$ such that for all prime numbers $p$, 
\begin{equation}\label{eq:main}
    \coindexPER_p(\XX(N, 1, \delta), \sigma)\ge \coindexPER_p(\ZZ, \sigma)+1.
\end{equation}
Fix such $N$ and $\delta$ through the whole section.

Recall that $\XX(N, k, \delta)=\{(x_n)_{n\in \Z}\in (\SS^N)^\Z: \rho_N(x_n, x_{n+k})\ge \delta, ~\forall n\in \Z \}.$ Define $$q(m)=\prod_{n=1}^{m}n$$ to be the product of first $m$ positive integers for $m\ge 1$. Obviously, we have $m\cdot q(m-1)=q(m)$ for $m\ge 2$.
Let $$X_m:=\XX(N, q(m), \delta)~\text{for}~ m\ge 1.$$  Here and below, we denote by $(X_m, T_m):=(X_m, \sigma)$ for convenience. For $m>1$, we define an equivariant continuous map $\theta_{m,m-1}: X_m \to (\SS^N)^\Z$ by
$$
(x_{k})_{k\in \Z} \mapsto \left(\sum_{i=0}^{m-1}x_{k+i\cdot q(m-1)} \right)_{k\in \Z}.
$$
As we have mentioned before, $\SS^N=(\R/2\Z)^N$ is an abelian group. Since
\begin{equation*}
\begin{split}
    &\rho_N\left(\sum_{i=0}^{m-1}x_{k+i\cdot  q(m-1)} ,   \sum_{i=0}^{m-1}x_{(k+q(m-1))+i\cdot  q(m-1)} \right)\\
    =&\rho_N\left(\sum_{i=0}^{m-1}x_{k+i\cdot  q(m-1)} ,   \sum_{i=1}^{m}x_{k+i\cdot  q(m-1)}\right)\\
    =&\rho_N\left(x_{k} ,  x_{k+ m\cdot q(m-1)} \right)=\rho_N\left(x_{k} ,  x_{k+ q(m)} \right),~\forall k\in \Z,
\end{split}
\end{equation*}
we obtain that the image of $X_m$ under $\theta_{m,m-1}$ is contained in $X_{m-1}$. For $m>n$, we define $\theta_{m,n}=\theta_{m,m-1}\circ \theta_{m-1,m-2} \circ \dots \theta_{n+1,n}$ to be  an equivariant continuous map from $X_m$ to $X_n$.

Fix $a=(a_k)_{k\in \Z}\in (\SS^N)^\Z.$ For $m\ge 2$, we define a map $\eta_{m-1,m}=\eta_{m-1,m}^a: X_{m-1} \to (\SS^N)^\Z$ by
$\eta_{m-1, m}((x_k)_{k\in \Z})=(y_k)_{k\in \Z}$ where 
\begin{equation*}
    y_k=
    \begin{cases}
    \bigbreak
    a_k~&\text{if}~0\le k\le (m-1)\cdot q(m-1)-1,\\
    \bigbreak
    x_{k-(m-1)q(m-1)}-\sum_{i=1}^{m-1}a_{k-i\cdot q(m-1)}~&\text{if}~(m-1)\cdot q(m)\le k\le q(m)-1\\
    \sum_{i=0}^{n-1} \left( x_{i\cdot q(m)+q(m-1)+j}-x_{i\cdot q(m)+j} \right) +y_j &\text{if}~k=n\cdot q(m)+j \\
    \bigbreak
   & ~\text{with}~n>0~\text{and}~0\le j\le q(m)-1,\\
    \sum_{i=n}^{-1} \left( x_{i\cdot q(m)+j}-x_{i\cdot q(m)+q(m-1)+j} \right) +y_j &\text{if}~k= n\cdot q(m)+j\\
    &~\text{with}~n<0~\text{and}~0\le j\le q(m)-1.
    \end{cases}
\end{equation*}
We show several properties of $\eta_{m-1,m}$ in the following lemmas.
\begin{lem}\label{lem:image of eta}
For $m\ge 2$, $\eta_{m-1,m}(X_{m-1})\subset X_m$.
\end{lem}
\begin{proof}
Let $x\in X_{m-1}$ and $y=\eta_{m-1,m}(x)$. Let $k=n\cdot q(m)+j$ with $n\in \Z$ and $0\le j\le q(m)-1$. We divide the proof in the following three cases according to the value of $n$.

\vspace{5pt}

\noindent Case 1. $n=1$. We have
$$
y_k-y_{k-q(m)}=x_{q(m-1)+j}-x_j+y_j-y_j=x_{q(m-1)+j}-x_{j}.
$$
Since $x\in X_{m-1}$, we have
$$
\rho_N(y_{k}, y_{k-q(m)})=\rho_N(x_{q(m-1)+j}, x_{j})\ge \delta.
$$
\\
Case 2. $n\ge 2$. A simple computation shows that
\begin{equation*}
    \begin{split}
        &y_{k}-y_{k-q(m)}\\
        =&\sum_{i=0}^{n-1} \left( x_{i\cdot q(m)+q(m-1)+j}-x_{i\cdot q(m)+j} \right) - \sum_{i=0}^{n-2} \left( x_{i\cdot q(m)+q(m-1)+j}-x_{i\cdot q(m)+j} \right)\\
        =&x_{(n-1) q(m)+q(m-1)+j}-x_{(n-1)\cdot q(m)+j}=x_{k-q(m)+q(m-1)}-x_{k-q(m)}.
    \end{split}
\end{equation*}
It follows that 
$$
\rho_N(y_{k}, y_{k-q(m)})=\rho_N(x_{k-q(m)+q(m-1)}, x_{k-q(m)})\ge\delta.
$$
\\
Case 3. $n\le 0$. Similarly to Case 1 and Case 2, we have
\begin{equation*}
    y_{k}-y_{k-q(m)}=x_{k-q(m)+q(m-1)}-x_{k-q(m)},
\end{equation*}
and consequently $\rho_N(y_{k}, y_{k-q(m)})\ge\delta.$

\vspace{5pt}

To sum up, we conclude that $y\in X_m$ and $\eta_{m-1,m}(X_{m-1})\subset X_m$.
\end{proof}

By Lemma \ref{lem:image of eta}, we see that $\eta_{m-1, m}$ is the map from $X_{m-1}$ to $X_m$. Moreover, we show in the following that $\eta_{m-1, m}$ is indeed a right inverse map of $\theta_{m, m-1}$.

\begin{lem}\label{lem:identity}
$\theta_{m,m-1}\circ\eta_{m-1,m}=\text{id}, \forall m\ge 2.$
\end{lem}
\begin{proof}
Let $x\in X_{m-1}$ and $y=\eta_{m-1,m}(x)$. Then we have
$$
\theta_{m,m-1}\circ\eta_{m-1,m}(x)=\theta_{m,m-1}(y)=\left(\sum_{i=0}^{m-1}y_{k+i\cdot q(m-1)} \right)_{k\in \Z}.
$$

If $0\le k\le q(m-1)-1$, then 
\begin{equation*}
    \begin{split}
        &\sum_{i=0}^{m-1}y_{k+i\cdot q(m-1)}\\
        =&\sum_{i=0}^{m-2}a_{k+i\cdot q(m-1)}+ \left(x_{k}-\sum_{i=1}^{m-1}a_{k+(m-1-i)q(m-1)}\right)
        =x_{k}.
    \end{split}
\end{equation*}

If $k=s\cdot q(m)+t\cdot q(m-1)+j>q(m-1)$ for $s>0$, $0\le t\le m-1$ and $0\le j\le q(m-1)-1$, then 
\begin{equation*}
    \begin{split}
      &\sum_{i=0}^{m-1}y_{k+i\cdot q(m-1)}\\
        =& \sum_{i=t}^{m-1}y_{s\cdot q(m)+i\cdot q(m-1)+j}+\sum_{i=0}^{t-1}y_{(s+1) q(m)+i \cdot q(m-1)+j}\\
        =& \sum_{i=t}^{m-1}\sum_{\ell=0}^{s-1} \left( x_{\ell \cdot q(m)+(i+1)q(m-1)+j}-x_{\ell\cdot q(m)+i\cdot q(m-1)+j} \right) \\
        &+\sum_{i=0}^{t-1}\sum_{\ell=0}^{s} \left( x_{\ell \cdot q(m)+(i+1)q(m-1)+j}-x_{\ell\cdot q(m)+i\cdot q(m-1)+j} \right) +\sum_{i=0}^{m-1}y_{i\cdot q(m-1)+j}\\
        =&\sum_{\ell=0}^{s-1}\sum_{i=0}^{m-1} \left( x_{\ell \cdot q(m)+(i+1)q(m-1)+j}-x_{\ell\cdot q(m)+i\cdot q(m-1)+j} \right) \\
        &+\sum_{i=0}^{t-1}\left( x_{s \cdot q(m)+(i+1)q(m-1)+j}-x_{s\cdot q(m)+i\cdot q(m-1)+j} \right)+ x_j\\
        =&\sum_{\ell=0}^{s-1} \left( x_{(\ell+1) \cdot q(m)+j}-x_{\ell\cdot q(m) +j} \right) 
        +\left( x_{s \cdot q(m)+t\cdot q(m-1)+j}-x_{s\cdot q(m)+j} \right)+ x_j\\
        =& x_{s\cdot q(m)+j}-x_j+x_{k}-x_{s\cdot q(m)+j} +x_j=x_k.
    \end{split}
\end{equation*}

If $k=s\cdot q(m)+t\cdot q(m-1)+j$ for $s<0$, $0\le t\le m-1$ and $0\le j\le q(m-1)-1$, then by the similar computation of the case where $s>0$, we have $\sum_{i=0}^{m-1}y_{k+i\cdot q(m-1)}=x_k$. This completes the proof.
\end{proof}

By Lemma \ref{lem:identity}, the map $\theta_{m,n}$ are surjective maps from $X_m$ to $X_n$ for all $m>n$. Therefore, denote by $\varprojlim (X_n,T_n)$ the inverse limit of $\{(X_n, T_n)\}_{n\in\N}$ via $(\theta_{m,n})_{n,m\in \N, m>n}$. 

\begin{lem}\label{lem:aperiodic}
The inverse limit $\varprojlim (X_n,T_n)$ is aperiodic.
\end{lem}
\begin{proof}
Suppose $\varprojlim (X_n,T_n)$ has a periodic point $y$ of period $n$. Let $\pi_n: \varprojlim (X_n,T_n) \to (X_n, T_n)$ be the natural projection. Then $x:=\pi_n(y)$ has the period $n$ in $(X_n, T_n)$. It follows from $n\mid q(n)$ that $x_0=x_{q(n)}$. This is a contradiction to the definition of $X_n$, that is, $|x_0-x_{q(n)}|>\delta$. Thus we conclude that  $\varprojlim (X_n,T_n)$ is aperiodic.
\end{proof}

We will show later that $\varprojlim (X_n,T_n)$ is desired for Theorem \ref{thm:finite mean dimension}. To this end, we need to do some preparation.
Lindenstrauss \cite[Lemma 3.3]{L99} introduced a topological dynamics version of the Rokhlin tower lemma in ergodic theory.
\begin{lem}\label{lem:Lindenstrauss}
Let $(X, T)$ be a dynamical system having the marker property. For each positive number $N$ there is a continuous function $\varphi: X \to \R$ such that the set
$$
E:=\{x\in X: \varphi(Tx)\not=\varphi(x)+1 \}
$$
satisfies that $E\cap T^{-n}E=\emptyset$ for all $1\le n\le N$.
\end{lem}
As it is pointed out by Gutman \cite[Theorem 7.3]{Gut15Jaworski}, the existence of a
continuous function $\varphi$ in Lemma \ref{lem:Lindenstrauss} is indeed equivalent to the marker property.  

Define
$$
\YY:=\{(x_n)_{n\in  \Z}\in  \SS^\Z: \forall n\in \Z, ~\text{either}~\rho(x_{n-1}, x_{n})=1 ~\text{or}~ \rho(x_{n}, x_{n+1})=1   \}.
$$
Concerning with the system $(\YY, \sigma)$, Tsukamoto, Tsutaya and Yoshinaga showed a necessary condition for the marker property \cite[Lemma 5.3]{tsukamoto2020markerproperty}. We give the proof here for completeness.

\begin{lem}\label{lem:to Y}
Let $(X,T)$ be a dynamical system having marker property. Then there
is an equivariant continuous map from $(X, T)$ to $(\YY, \sigma)$.
\end{lem}
\begin{proof}
By Lemma \ref{lem:Lindenstrauss}, there is a continuous function $\varphi: X \to \R$ such that the set
$$
E:=\{x\in X: \varphi(Tx)\not=\varphi(x)+1 \}
$$
satisfies that $E\cap T^{-1}E=\emptyset$. Let $P: \R \to \SS$ be the natural projection. Set $\phi=P\circ\varphi: X\to \SS$. Notice that 
\begin{equation}\label{eq:11}
    x\notin E \Longrightarrow \rho(\phi(Tx), \phi(x))=1.
\end{equation}
Define an equivariant continuous map $f: X\to \SS^\Z$ by
$$
x\mapsto (\phi(T^nx))_{n\in \Z}.
$$
Fix $x\in X$ and $ n\in \Z$. Since $E\cap T^{-1}E=\emptyset$, we get that either $T^nx\notin E$ or $T^{n+1}x\notin E$. It follows from \eqref{eq:11} that either $\rho(\phi(T^{n+1}x), \phi(T^nx))=1$ or $\rho(\phi(T^{n+2}x), \phi(T^{n+1}x))=1$. Since $x$ and $n$ are chosen arbitrarily, we obtain that the image of $X$ under $f$ is included in $\YY$. This completes the proof.

\end{proof}

Recall that $\eta_{m-1,m}$ is the map from $X_{m-1}$ to $X_m$ for $m\ge 2$.
Since we have  
\begin{equation*}
    \begin{split}        
    &d(\eta_{m-1,m}(x), \eta_{m-1,m}(y))\\
    \le  & \sum_{k=(m-1)q(m-1)}^{q(m)-1} \frac{1}{2^k} \rho_N(x_{k-(m-1)q(m-1)}, y_{k-(m-1)q(m-1)})
        \\
        & +\sum_{n=1}^{+\infty} \sum_{j=0}^{q(m)-1} \frac{1}{2^{n\cdot q(m)+j}} \sum_{i=0}^{n-1}(\rho_N(x_{i\cdot q(m)-q(m-1)+j}, y_{i\cdot q(m)-q(m-1)+j})\\
        &+\rho_N(x_{i\cdot q(m)+j}, y_{i\cdot q(m)+j}))
        + \sum_{n=-1}^{-\infty} \sum_{j=0}^{q(m)-1} \frac{1}{2^{|n\cdot q(m)+j|}} \cdot \\ &\sum_{i=n}^{-1}(\rho_N(x_{i\cdot q(m)+j}, y_{i\cdot q(m)+j})+\rho_N(x_{i\cdot q(m)+q(m-1)+j}, y_{i\cdot q(m)+q(m-1)+j})),
    \end{split}
\end{equation*}
the map $\eta_{m-1,m}$ is continuous for each $m\ge 2$. We define a continuous map $\gamma_m: X_m \to \varprojlim (X_n,T_n)$ by
$$
x\mapsto (\theta_{m,1}(x), \theta_{m,2}(x), \dots, \theta_{m,m-1}(x), x, \eta_{m,m+1}(x), \eta_{m,m+2}(x), \dots ),
$$
where $\eta_{m,m+n}(x)=\eta_{m+n-1,m+n}\circ\cdots\eta_{m+1,m+2}\circ\eta_{m,m+1}(x)$.

\begin{lem}\label{lem:to Z}
Let $(X_n, T_n)$ be defined as above.
Suppose  there
is an equivariant continuous map $f: \varprojlim (X_n,T_n) \to (\YY, \sigma)$. Then
there exists an integer $M$ and an equivariant continuous map
 $$g: (X_M, T_M) \longrightarrow  (\ZZ, \sigma). $$ 
\end{lem}
\begin{proof}
For convenience,  we denote the inverse limit by $(\XX, T)=\varprojlim (X_n,T_n)$. Let $\pi_m: \XX\to X_m$ be the natural projection for $m\in \N$. Let $P_1: \SS^\Z \to \SS$ be the projection on $0$-th coordinate. 
Define $\phi=P_1\circ f: \XX \to \SS$. Then $f(x)=(\phi(T^nx))_{n\in \Z}$ for any $x\in \XX$.

Notice that there exists an integer $M>0$ such that 
\begin{equation}\label{eq:3}
    \pi_M(x)=\pi_M(y) \Longrightarrow \rho(\phi(x), \phi(y))<\frac{1}{4}.
\end{equation}
Define a continuous map $\varphi=\phi\circ \gamma_M: X_M \to \SS$ where $\gamma_M: X_M\to \XX$ is defined as above. Now we define an equivariant continuous map $g: X_M \to \SS^\Z$ by 
$$
x \mapsto (\varphi(T_M^n(x)))_{n\in \Z}.
$$
Since $\pi_M\circ \gamma_M={\rm id}$ and $\pi_M\circ T=T_M\circ \pi_M$, it follows from \eqref{eq:3} that 
\begin{equation}\label{eq:4}
    \rho\left(\phi(\gamma_M(T_M^n(x)) ), \phi(T^n(\gamma_M(x)) )\right)<\frac{1}{4}, \forall n\in \Z.
\end{equation}
Fix $x\in X_M$ and $n\in \Z$. By definitions of $\YY$ and $f$, there exists an $i\in \{0,1\}$ such that
\begin{equation}\label{eq:5}
    \rho(\phi(T^{n+i}(\gamma_M(x))),\phi(T^{n+i+1}(\gamma_M(x))) )=1.
\end{equation}
Combing \eqref{eq:5} with \eqref{eq:4}, we obtain that
\begin{equation*}
\begin{split}
    &\rho\left(\phi(\gamma_M(T_M^{n+i}x) ), \phi(\gamma_M(T_M^{n+i+1}x) )\right)\\
    \ge 
    &~ \rho(\phi(T^{n+i}(\gamma_M(x))),\phi(T^{n+i+1}(\gamma_M(x))) ) \\
    &\quad - \rho\left(\phi(\gamma_M(T_M^{n+i}x) ), \phi(T^{n+i}(\gamma_M(x)) )\right)\\
    &\qquad -\rho\left(\phi(\gamma_M(T_M^{n+i+1}x) ), \phi(T^{n+i+1}(\gamma_M(x)) )\right)\\
    \ge&~ 1-\frac{1}{4}-\frac{1}{4}=\frac{1}{2}.
\end{split}
\end{equation*}
Since $\varphi=\phi\circ \gamma_M$, we have that $$\rho\left(\varphi(T_M^{n+i}x), \varphi(T_M^{n+i+1}x) \right)\ge \frac{1}{2}.$$
By definition of $g$ and arbitrariness of $n$ and $x$, we conclude that the image of $X_M$ under $g$ is contained in $\ZZ$. This completes the proof. 
\end{proof}

Now we present a result for general inverse limits.

\begin{prop}\label{prop:mean dimension of inverse limit}
Let $\{(X_n, T_n)\}$ be a sequence of dynamical system. Let $\varprojlim (X_n,T_n)$ be the inverse limit of  $\{(X_n, T_n) \}_{n\in \N}$ via factor maps $(\tau_{m,n})_{m,n\in \N, m>n}$. Then
$$
\mdim ( \varprojlim (X_n,T_n))\le \sup_{n} \mdim (X_n, T_n).
$$
\end{prop}
\begin{proof}
Let $\pi_i: \varprojlim (X_n,T_n) \to X_i $ be the projections.
Pick a finite open cover $\AA$ of $\varprojlim (X_n,T_n)$. Then there exists $m=m(\AA)>0$ and finite open covers $\AA_i$ of $X_i$ for $1\le i\le m$ such that 
$$
\bigvee_{i=1}^m \pi_i^{-1}(\AA_i) \succ \AA. 
$$
Notice that 
$$
\bigvee_{i=1}^m \pi_i^{-1}(\AA_i)=\bigvee_{i=1}^m \pi_m^{-1}(\tau_{m,i}^{-1}(\AA_i))= \pi_m^{-1}\left(\bigvee_{i=1}^m\tau_{m,i}^{-1}(\AA_i)\right).
$$
Set $\BB:=\bigvee_{i=1}^m\tau_{m,i}^{-1}(\AA_i)$ which is a finite open cover of $X_m$. Then we have $\pi_m^{-1}(\BB)$ refines $\AA$. It follows that
\begin{equation*}
    \begin{split}
        \frac{D(\vee_{i=0}^{N} T^{-i}\AA )}{N} &\le \frac{D(\vee_{i=0}^{N} T^{-i} \pi_m^{-1}(\BB) )}{N}\\
&=\frac{D(\pi_m^{-1} (\vee_{i=0}^{N} T_m^{-i} (\BB)) )}{N}
\le  \frac{D(\vee_{i=0}^{N} T_m^{-i}\BB )}{N}.
    \end{split}
\end{equation*}
Therefore, it implies that $\mdim(\varprojlim (X_n,T_n))\le \sup_{n} \mdim (X_n, T_n).$ This completes the proof.
\end{proof}

Now we prove the main result in this section.
\begin{proof}[Proof of Theorem \ref{thm:finite mean dimension}]
Let $(X_m, T_m)$ be the dynamical systems defined as above. Since $X_m\subset (\SS^N)^\Z$, for every $m\in \N$, we have $$\mdim(X_m, T_m)\le \mdim((\SS^N)^\Z, \sigma)\le N$$ where the last inequality is due to \cite[Proposition 3.1]{LinWeiss2000MeanTopologicalDimension}.
Let $(\XX, T):= \varprojlim (X_n,T_n)$. By Lemma \ref{lem:aperiodic}, the dynamical system $(\XX, T)$ is aperiodic. By Proposition \ref{prop:mean dimension of inverse limit}, we have $\mdim(\XX, T)\le N$. Suppose $(\XX, T) $ has the marker property. By Lemma \ref{lem:to Y}, there
is an equivariant continuous map $f: (\XX, T) \to (\YY, \sigma)$. Then by Lemma \ref{lem:to Z}, there exists $M>0$ and
 an equivariant continuous map $g: (X_M, T_M) \to (\ZZ, \sigma)$. It follows from Corollary \ref{cor:basic property} (1) that 
$$
\coindexPER_p(X_M,T_M) \le \coindexPER_p(\ZZ, \sigma),
$$
for all prime numbers $p>M$.
By Lemma \ref{lem:m are the same coindex}, we get
$$
\coindexPER_p(X_1,T_1) \le \coindexPER_p(\ZZ, \sigma),
$$
for all prime numbers $p>M$. This is a contradiction to the choice of $N$ and $\delta$, that is, \eqref{eq:main}. Thus $(\XX, T)$ is an aperiodic dynamical system having finite mean dimension which does not have the marker property. We complete the proof.
\end{proof}


\section{Marker property and small mean dimension}\label{sec:Marker property and suspension flow}

Let $(X, T)$ be a dynamical system. Let $\Z_n:=\Z/n\Z$ be a finite abelian group for $n\ge 1$. We define the continuous map $T_n: X\times \Z_n \to X \times \Z_n$ by 
$$
T_n(x, k)=
\begin{cases}
\vspace{5pt}
(x, k+1)~&\text{if}~0\le k\le n-2,\\
(Tx,0)~&\text{if}~k=n-1.
\end{cases}
$$ 
We call $(X\times \Z_n, T_n)$ the {\it $\frac{1}{n}$-time system} over $(X,T)$. Actually, the dynamical system $(X\times \{0\}, T_n^n)$ is topologically conjugate to $(X, T)$. It follows that $\mdim(X\times \Z_n, T_n)=\frac{1}{n}\mdim(X,T)$.

Now we investigate the relation between the marker property of a dynamical system and the one of its $\frac{1}{n}$-time system.
\begin{prop}\label{prop:suspension flow and marker property}
A dynamical system $(X, T)$ has the marker property if and only if $(X\times \Z_n, {T}_n)$ has the marker property.
\end{prop}
\begin{proof}
Suppose $(X, T)$ has the marker property. Let $N\ge 1$. Let $U\subset X$ be an open $N$-marker, i.e. $U\cap T^{k} U =\emptyset$ for $0<k\le N$ and $X=\cup_{k\in \Z} T^kU$. It is easy to see that $U\cap T_n^{k} U =\emptyset$ for $0<k\le nN$ and $X\times \Z_n=\cup_{k\in \Z} T_n^kU$. Thus $U$ is an open $nN$-marker of $(X\times \Z_n, {T}_n)$. As $N$ is chosen arbitrarily, we get that $(X\times \Z_n, {T}_n)$ has the marker property.

Now suppose $(X\times \Z_n, {T}_n)$ has the marker property. Since $(X, T)$ is topologically conjugate to $(X\times \{0\}, T_n^n)$, it is sufficient to show that $(X\times \{0\}, T_n^n)$ has the marker property.
Fix $N>0$. Let $W$ be an open $nN$-marker of $(X\times \Z_n, {T}_n)$. Let 
 $$U:=\left(\cup_{j=0}^{n-1}T_n^j(W)\right)\cap (X\times \{0\}),$$
 which is an open set in $X\times \{0\}$. Notice that 
 $$T_n^{kn}U=\left(\cup_{j=0}^{n-1}T_n^{kn+j}(W)\right)\cap (X\times \{0\}).$$
 By assumption, we have that $T_n^{i}(W) \cap T_n^{j}(W)\not=\emptyset$ whenever $i\not=j$ and $|i-j|\le nN$.
 It follows that
 $$
 U \cap T_n^{kn}U=\left(\cup_{j=0}^{n-1}T_n^j(W)\cap \cup_{j=0}^{n-1}T_n^{kn+j}(W)  \right)\cap (X\times \{0\})=\emptyset,
$$
for $0<k\le N-1$.
Moreover, 
$$
\cup_{k\in \Z} T_n^{kn}U =\left(\cup_{j\in \Z}T_n^j(W)\right)\cap (X\times \{0\}) =(X\times \Z_n)\cap (X\times \{0\})=X\times \{0\}.
$$
It means that $U$ is a $(N-1)$-marker of $(X\times \{0\}, T_n^n)$. Since $N$ is chosen arbitrarily, the system $(X\times \{0\}, T_n^n)$ has the marker property.
\end{proof}

Now we present the main result in this section.
\begin{thm}[=Theorem \ref{main thm}]\label{main thm 2}
For any $\eta>0$, there exists an aperiodic dynamical system with mean dimension small than $\eta$ which does not have the marker property. 
\end{thm}
\begin{proof}
By Theorem \ref{main thm}, there exists an aperiodic finite mean dimensional dynamical system $(\XX, T)$ which does not have the marker property. Let $n\ge 1$ and let $(\XX\times \Z_n, {T}_n)$ be the dynamical system defined as above. By Lemma \ref{prop:suspension flow and marker property}, the dynamical system  $(\XX\times \Z_n, {T}_n)$ does not have the marker property. Notice that
$$
\mdim(\XX\times \Z_n, {T}_n)=\frac{1}{n} \mdim(\XX, T).
$$
We complete the proof by taking $n$ arbitrarily large.
\end{proof}

\section{Open problem}\label{sec:open problem}

In this section, we discuss open problems related to the results of the current paper. Firstly, we present a corollary of Theorem \ref{main thm 2}.

\begin{cor}\label{cor:any mean dimension example}
For any number $\eta\in (0, +\infty]$, there exists an aperiodic dynamical system with mean dimension equal to $\eta$ which does not have the marker property.
\end{cor}
\begin{proof}
Let $\eta\in (0, +\infty]$. By Theorem \ref{main thm 2},  there exists an aperiodic dynamical system $(X,T)$ with mean dimension small than $\eta$ which does not have the marker property. Using the construction by Lindenstrauss and Weiss \cite[Page 10-11]{LinWeiss2000MeanTopologicalDimension} (see also \cite[Section 3]{lindenstraussTsukamoto2014mean}), there is a minimal dynamical system $(Y,S)$ with mean dimension equal to $\eta$.
By definition of mean dimension, it is easy to check that $\mdim((X, T) \cup (Y,S) )=\max\{ \mdim(X, T), \mdim(Y,S) \}=\eta$. Since $(X, T) $ and $(Y,S)$ are both aperiodic, the system $(X, T) \cup (Y,S)$ is also aperiodic. Moreover, $(X, T) \cup (Y,S)$ does not have the marker property. Indeed, if $(X, T) \cup (Y,S)$ has the marker property, then so do both $(X, T) $ and $(Y,S)$. It is a contradiction because $(X, T) $ does not have the marker property. Thus the system $(X, T) \cup (Y,S)$ is what we desire. 
\end{proof}


As it was discussed in \cite{tsukamoto2020markerproperty}, one need certain additional condition besides the aperiodicity to guarantee the marker property. Due to Corollary \ref{cor:any mean dimension example}, the condition of zero mean dimension seems plausible. Here we restate \cite[Conjecture 7.4]{tsukamoto2020markerproperty} as follows. 
\begin{conj}\label{conj:zero mean dimension}
An aperiodic dynamical system with zero mean dimension has the
marker property.
\end{conj}

Gutman \cite[Theorem 6.1]{Gut15Jaworski} proved that every finite dimensional aperiodic dynamical system has the marker property. Thus it reduces Conjecture \ref{conj:zero mean dimension} to the case that aperiodic dynamical systems have zero mean dimension and are infinite dimensional. However, it seems difficult and is widely open in such case.

A dynamical system is said to have the {\it small boundary property} if for any point one can find a neighborhood whose boundary is a null set. Lindenstrauss \cite{L99} proposed the following conjecture. As it is pointed out in \cite{tsukamoto2020markerproperty}, Conjecture \ref{conj:zero mean dimension} is indeed equivalent to the following conjecture proposed by Lindenstrauss.
\begin{conj}\label{conj:Lindenstrauss}
An aperiodic dynamical system with zero mean dimension has the
small boundary property.
\end{conj}
Lindenstrauss \cite[Theorem 6.2]{L99} proved that if a 
dynamical system with
 zero mean dimension has the marker property then it has the small boundary property. Nevertheless, Conjecture \ref{conj:Lindenstrauss} is still open in general. 
 


\section*{Acknowledgement} 
We thank Masaki Tsukamoto and Yonatan Gutman for valuable discussion and remarks.

\bibliographystyle{alpha}
\bibliography{universal_bib}

\end{document}